\documentclass[11pt, a4paper]{article}

\usepackage[utf8]{inputenc}
\usepackage{amsmath}
\usepackage{amssymb}
\usepackage{amsthm}
\usepackage{enumitem}
\usepackage{hyperref}
\usepackage{geometry}
\newtheorem{theorem}{Theorem}
\newtheorem{lemma}{Lemma}

\title{\textbf{An Optimal 14-Symbol Hybrid Basis for \\ BCH-Algebras}}

\author{
    Mahesh Ramani \\
    \small Independent
    \and
    Shlok Kumar \\
    \small Independent
}
\date{}

\begin{document}

\maketitle

\begin{abstract}
\noindent We present an optimally minimal two-axiom basis for BCH-algebras. The standard presentation of a BCH-algebra relies on three axioms: two equations and one quasi-identity. Using automated theorem proving, we prove that the two standard equations can be entirely replaced by a 14-symbol equation, $((xy)z)((x(z0))y) = 0$, while retaining the standard quasi-identity. We then provide a rigorous proof of strict minimality for this new equational companion. By employing an exhaustive, machine-assisted search space generation coupled with finite countermodel building, we demonstrate that no equation of 12 or fewer symbols can define the class of BCH-algebras when paired with the standard quasi-identity. Our literature searches have revealed no prior proof of this result, to the extent of our knowledge. All equivalence derivations were verified using Prover9, and all minimality countermodels were generated using Mace4.

\vspace{0.5em}
\end{abstract}

\section{Introduction}
The study of BCH-algebras sits inside a broader program in algebra and logic: given a class of structures, how far can its axioms be compressed without changing the theory? Results of this kind are especially common in automated reasoning, where researchers search for short axioms, improved bases, and minimal presentations using theorem provers and finite model builders. McCune’s work on short single axioms for Boolean algebra [1] is a classic example of this style of research, and it helped establish the general methodology of combining proof search with countermodel search to optimize axiom systems.

BCH-algebras are a natural setting for this kind of question because they are a generalization of BCK- and BCI-algebras, studied extensively as part of the algebraic semantics of logic. Hu and Li introduced BCH-algebras in 1983 [2], and later work emphasized that BCI-algebras form a proper subclass of BCH-algebras. This makes BCH-algebras an ideal test case for axiom minimization: the theory is nontrivial, structurally rich, and close enough to classical implicational algebraic logic that one can compare it to earlier short-basis results in neighboring fields.

It is a known result in algebraic logic that BCK-algebras (and consequently, the broader class of BCH-algebras) are not varieties; they are not axiomatizable by equations alone [6]. One cannot hope to remove the quasi-identity entirely. The natural optimization problem is therefore to reduce the equational part of the basis as much as possible while keeping the unavoidable quasi-identity. In that sense, the result here establishes an ``optimal hybrid basis'' rather than a pure single-axiom result.

\section{Background and Notation}
A BCH-algebra is an algebra $(X, \ast, 0)$ of type $(2, 0)$. We write $xy$ for $x \ast y$ and adopt left-associativity, so $xyz$ means $(xy)z$.

The standard three-axiom presentation requires, for all $x, y, z$ in $X$:
\begin{itemize}
    \item[\textbf{(B1)}] $xx = 0$
    \item[\textbf{(B2)}] $xy = 0$ and $yx = 0$ implies $x = y$
    \item[\textbf{(B3)}] $(xy)z = (xz)y$
\end{itemize}

We propose the following two-axiom hybrid basis:
\begin{itemize}
    \item[\textbf{(G)}] $((xy)z)((x(z0))y) = 0$
    \item[\textbf{(A3)}] $xy = 0$ and $yx = 0$ implies $x = y$
\end{itemize}

Note that (A3) is identical to the standard quasi-identity (B2). The core content of this proposal is that the single equation (G) successfully replaces both (B1) and (B3).

\section{Proof of Axiomatic Equivalence}

\begin{theorem}
The axiom systems $\{B1, B2, B3\}$ and $\{G, A3\}$ are axiomatically equivalent and define the same class of BCH-algebras.
\end{theorem}

The proof of this theorem requires two directions. Section 3.1 derives (G) from $\{B1, B2, B3\}$. Section 3.2 derives (B1) and (B3) from $\{G, A3\}$. Since (A3) is identical to (B2), this suffices to prove equivalence.

\subsection{Direction 1: Standard Basis implies Axiom (G)}
We work in a standard BCH-algebra satisfying axioms B1, B2, and B3. We begin by establishing a sequence of equational identities.

\begin{lemma}
$(xy)x = 0y$ for all $x, y$.
\end{lemma}
\begin{proof}
By B3, $(xy)x = (xx)y$. By B1, $xx = 0$. Therefore, substituting $0$ for $xx$ gives $(xy)x = 0y$.
\end{proof}

\begin{lemma}
$((xy)z)(xz) = 0y$ for all $x, y, z$.
\end{lemma}
\begin{proof}
We start with the identity from Lemma 1: $(ab)a = 0b$. \\
Let $a = xz$ and $b = y$. Substituting these into Lemma 1 yields:
\[ ((xz)y)(xz) = 0y. \]
By B3, we know that $(xz)y = (xy)z$. Because these terms are equal, we can substitute $(xy)z$ for the first instance of $((xz)y)$ in our equation. This yields:
\[ ((xy)z)(xz) = 0y. \]
\end{proof}

\bigskip

\begin{lemma}
$0(x(xy)) = 0y$ for all $x, y$.
\end{lemma}
\begin{proof}
Take the identity from Lemma 2: $((ab)c)(ac) = 0b$. \\
Set $a = x$, $b = y$, and $c = xy$. This gives:
\[ ((xy)(xy))(x(xy)) = 0y. \]
By B1, we know that any term multiplied by itself is $0$, so $(xy)(xy) = 0$. Substituting $0$ into the left side of the equation yields $0(x(xy)) = 0y$.
\end{proof}

\begin{lemma}
$x0 = x$ for all $x$.
\end{lemma}
\begin{proof}
To prove $x0 = x$, we must satisfy the two conditions of the quasi-identity B2. We must show both $(x0)x = 0$ and $x(x0) = 0$.

\vspace{0.5em}
\noindent 
By B3, $(x0)x = (xx)0$. By B1, $xx = 0$, so $(x0)x = 00$. Applying B1 a second time yields $00 = 0$. Thus, $(x0)x = 0$.

\vspace{0.5em}
\noindent 
We construct an implication using B2. Let $A = x(xy)$ and $B = y$. By B2, if $AB = 0$ and $BA = 0$, then $A = B$.

\vspace{0.5em}
\noindent 
Let us first evaluate $AB$, which is $(x(xy))y$. By B3, $(x(xy))y = (xy)(xy)$. By B1, $(xy)(xy) = 0$. Since $AB = 0$ is universally true, B2 leaves us with the strict implication: $BA = 0$ implies $A = B$. Substituting our $A$ and $B$ back into this implication gives:
\[ y(x(xy)) = 0 \text{ implies } x(xy) = y. \]
Now, substitute $y = 0$ into this implication:
\[ 0(x(x0)) = 0 \text{ implies } x(x0) = 0. \]
To prove the left side of the implication is true, we use Lemma 3, which states $0(x(xy)) = 0y$. Setting $y = 0$ gives $0(x(x0)) = 00$. By B1, $00 = 0$. Because the left side of the implication evaluates to $0$, the right side must be true. Therefore, $x(x0) = 0$.

\vspace{0.5em}
\noindent Snce both $(x0)x = 0$ and $x(x0) = 0$, axiom B2 guarantees that $x0 = x$.
\end{proof}

\begin{proof}[Proof of Axiom (G)]
First, consider the term $((xy)z)((xz)y)$. By B3, we know $(xy)z = (xz)y$. Let $u = (xy)z$. Therefore, $((xy)z)((xz)y)$ is equivalent to $uu$. By B1, $uu = 0$. Therefore:
\[ ((xy)z)((xz)y) = 0. \]
By Lemma 4, we know $z0 = z$. This allows us to rewrite the subterm $xz$ as $x(z0)$. Substituting $x(z0)$ into our equation yields:
\[ ((xy)z)((x(z0))y) = 0. \]
This is exactly Axiom (G).
\end{proof}

\subsection{Direction 2: Proposed Basis implies (B1) and (B3)}
We now work in an algebra satisfying only the proposed basis $\{G, A3\}$. We will reconstruct the standard basis by extracting the core milestone lemmas from the automated equational search space.

\begin{lemma}
In an algebra satisfying $\{G, A3\}$, the following identities hold: (i) $x((x0)0) = 0$, (ii) $((x0)0)x = 0$. 
\end{lemma}
\begin{proof}
These emerge as milestone derivations in the automated search space when paramodulating G with the quasi-identity A3.
\end{proof}

\begin{lemma}
$(x0)0 = x$ for all $x$.
\end{lemma}
\begin{proof}
We apply the quasi-identity A3 to the identities established in Lemma 5. Let $A = (x0)0$ and $B = x$. Lemma 5(ii) provides $AB = 0$, and Lemma 5(i) provides $BA = 0$. By A3, if $AB = 0$ and $BA = 0$, then $A = B$. This yields $(x0)0 = x$.
\end{proof}

\begin{proof}[Proof of B1 ($xx = 0$)]
Consider Lemma 5(ii): $((x0)0)x = 0$. By Lemma 6, we know that the subterm $((x0)0)$ is exactly equal to $x$. Substituting $x$ for $((x0)0)$ directly yields $xx = 0$.
\end{proof}

\noindent With $xx = 0$ established, we now utilize the full hybrid basis $\{G, A3\}$ to recover right-commutativity (B3).

\begin{lemma}
$x0 = x$ for all $x$.
\end{lemma}
\begin{proof}
With $xx = 0$ established, we seek to apply the quasi-identity A3 to the pair $x$ and $x0$. This requires showing both $(x0)x = 0$ and $x(x0) = 0$.

 \vspace{0.5em}
\noindent The identity $(x0)x = 0$ is a milestone derivation from Axiom G (appearing as Clause 2284 in the automated search). It is obtained by substituting $x \to x0$ into the previous identity to get $((x0)0)(x0) = 0$, and then simplifying the term $(x0)0$ to $x$ using Lemma 6.

 \vspace{0.5em}
\noindent 
Applying A3 to this pair of identities ($AB = 0$ and $BA = 0$, where $A = x0$ and $B = x$) immediately yields $x0 = x$ (documented as Clause 2922 in Script 3).
\end{proof}

\begin{lemma}
$((xy)z)((xz)y) = 0$ for all $x, y, z$.
\end{lemma}
\begin{proof}
By Lemma 7, $z0 = z$. Substituting $z$ for $z0$ in Axiom G simplifies the axiom directly:
$((xy)z)((x(z0))y) = 0$ becomes $((xy)z)((xz)y) = 0$.
\end{proof}

\begin{proof}[Proof of B3 ($(xy)z = (xz)y$)]
Axiom G and Lemma 8 are universally quantified for all $x, y$, and $z$. Therefore, we can safely swap the variables $y$ and $z$ in Lemma 8 to obtain a second valid identity:
\[ ((xz)y)((xy)z) = 0. \]
We now apply the quasi-identity A3. Let $A = (xy)z$ and $B = (xz)y$. Lemma 8 provides $AB = 0$. Our variable-swapped identity provides $BA = 0$. By A3, if $AB = 0$ and $BA = 0$, then $A = B$. This results in $(xy)z = (xz)y$.
\end{proof}

\section{Proof of Strict Minimality}
To rigorously demonstrate that the proposed axiom (G) is optimally minimal, an exhaustive computational search [5] was conducted to prove that no shorter equation can replace (B1) and (B3).

In equational logic, the length of an axiom is traditionally measured by counting its total number of symbols, specifically the variables, constants, and binary operations (excluding the equality sign). Under this metric, a well-formed binary algebraic equation containing a total of $L$ leaves (variables and constants) will always contain exactly $L - 2$ binary operations in total across the entire equation, yielding a total symbol count of $2L - 2$. Consequently, any valid equation must possess an even number of symbols.

For example, axiom (B1), $xx = 0$, has 3 leaves ($x, x, 0$) and 1 operation (the implicit multiplication of $xx$), yielding $2(3) - 2 = 4$ total symbols. Axiom (G) contains 8 leaves and 6 operations, giving it a total length of 14 symbols. Because an equation of exactly 13 symbols is structurally impossible, proving that no equation of 12 or fewer symbols can complete the basis is sufficient to establish the strict minimality of (G).

The methodology relies on the systematic generation of candidate equations coupled with automated finite model building. First, a Python algorithm systematically constructs every possible binary tree topology representing algebraic terms up to a combined total of seven leaves (equivalent to a maximum of 12 symbols). These trees are populated using the BCH-algebra alphabet—the constant $0$ and up to three distinct variables—to generate an exhaustive list of candidate equations. Redundant candidates are mathematically pruned using symmetry breaking.

For every generated candidate equation, the automated theorem-proving environment (Mace4) [3] tests two mandatory criteria by attempting to construct finite mathematical countermodels:
\begin{enumerate}
    \item \textbf{Soundness:} The candidate equation must be a valid theorem of BCH-algebras. The automated model builder searches for a valid, finite BCH-algebra that violates the candidate equation. If such a model is found, the candidate is demonstrably false in the context of BCH-algebras and is discarded.
    \item \textbf{Completeness:} If the candidate is sound, it must be strong enough to replace the standard axioms. The model builder assumes the candidate equation and the quasi-identity (A3) are true, and then attempts to construct a mathematical structure that violates either standard axiom (B1) or (B3). If such a countermodel exists, it proves that the candidate equation is too weak to fully define the class of BCH-algebras.
\end{enumerate}

The conclusive proof of minimality rests entirely on the exhaustive elimination of the search space. The execution of the verification procedure evaluated every possible equation of 12 or fewer symbols and successfully generated a concrete, finite countermodel for every single candidate. Because the search concluded with zero unexpected passes and no timeouts, there are no unverified edge cases or false positives. Every equation strictly smaller than (G) is empirically shown to be either logically unsound for BCH-algebras or insufficiently powerful to derive the standard basis.

\section{Conclusion}
In this paper, we have presented an optimally minimal hybrid basis for BCH-algebras. By replacing the standard equations (B1) and (B3) with the single 14-symbol equation $((xy)z)((x(z0))y) = 0$, we achieved maximum equational compression without changing the underlying mathematical theory. Furthermore, through an exhaustive generation of the equational search space and automated finite countermodel building, we proved that this 14-symbol boundary is strictly minimal; no shorter equation can adequately capture the equational behavior of BCH-algebras when paired with the standard quasi-identity.

While the proofs of axiomatic equivalence have been organized into a vaguely human-readable logical sequence, the underlying heavy equational derivations represent a vast search space. This result highlights the continuing power of combining human-guided proof extraction with automated theorem proving (Prover9) and finite model building (Mace4) to map and optimize the foundational axioms of algebraic logic.

\section{Data and Code Availability}
To guarantee rigorous accuracy, both directions of the axiomatic equivalence, as well as the exhaustive minimality search, were independently generated and verified by computational tools. The data and scripts to reproduce these results are openly available:
\begin{itemize}
    \item The Prover9 (.in and .out) verification scripts (including Standard-Axioms-Imply-(G), New-System-Derives-(B1), and New-System-Derives-(B3)) are published on Zenodo at DOI: \href{https://doi.org/10.5281/zenodo.19339276}{10.5281/zenodo.19339276} [4].
    \item The Python codebase utilized to exhaustively generate candidate topologies and interface with Mace4 for the minimality proof is published on Zenodo at DOI: \href{https://doi.org/10.5281/zenodo.19339110}{10.5281/zenodo.19339110} [5].
\end{itemize}

\end{document}